\documentclass[11pt]{amsart}
\usepackage[utf8]{inputenc}
\usepackage[english]{babel}
\usepackage{graphicx}
\usepackage{amssymb}
\usepackage{amsmath}
\usepackage[mathscr]{euscript}
\usepackage{accents}
\usepackage{color}

\usepackage{tikz}
\usetikzlibrary{decorations.markings}

\usepackage{microtype} 

\theoremstyle{plain}	
\newtheorem{theorem}{Theorem}[section]
\newtheorem{proposition}[theorem]{Proposition}

\theoremstyle{definition} 
\newtheorem{definition}[theorem]{Definition}

\newcommand{\Lbrack}{[\![}
\newcommand{\Rbrack}{]\!]}

\title{Bicategories for TQFTs with Defects with Structure}
\author{I.\ J.\ Lee and D.\ N.\ Yetter}

\address[leei@rowan.edu]{I.\ J.\ Lee \\ Mathematics Department \\ Rowan University \\ Glassboro, NJ 08028 \\U.S.A}
\address[dyetter@math.ksu.edu]{D.\ N.\ Yetter \\ Department of Mathematics \\ Cardwell Hall \\ Kansas State University \\ Manhattan, KS 66506 \\ U.S.A}

\begin{document}
\tikzset{midarrow/.style={
		decoration={markings,
			mark= at position 0.5 with {\arrow{#1}},
		},
		postaction={decorate}
	}
}

\begin{abstract}
We provide a description of adequate categorical data to give a Turaev-Viro type state-sum construct of invariants of 3-manifolds with a system of defects, generalizing the Dijkgraaf-Witten type invariants of our earlier work.  We term the defects in our construction defects-with-structure because algebraic data associated to them is in general richer than a module category over the spherical fusion category from which the theory is constructed when no defect is present.
\end{abstract}



\maketitle

\section{Introduction}

It is well known that spherical fusion categories provide the initial data for a state-sum construction of 3-dimensional topological quantum field theories (TQFTs) generalizing that given by Turaev and Viro \cite{TV} using the fusion category constructed from the quantum group $U_q sl(2)$ for $q$ a root of unity.  The generalization to arbitrary ribbon fusion categories (which are necessarily spherical) was given by Yetter \cite{Y}, with the general case for arbitrary spherical fusion categories finally given given by Barrett and Westbury \cite{BW}, who realized that a braiding and ribbon structure were superfluous and the only the extra property of the coincidence of left and right traces was needed. 

In \cite{LY1} the authors extended Wakui's \cite{Wa} state-sum construction of finite-gauge-group Dijkgraaf-Witten theory, which retrospectively is a special case of the Barrett and Westbury general construction, to give invariants of 3-manifolds equipped with a link, a surface, or a link with a Seifert surface.  The construction regarded the space equipped with subspaces as very simple stratified spaces and used the stratification to make the 3-manifold into a directed space in the sense of Grandis \cite{Gr1, Gr2} (see also \cite{LY2}), replacing the finite group with a finite category equipped with a conservative functor to a poset encoding the adjacencies of the (dimensions of) strata.

It is the purpose of this paper to provide the generalization of this earlier work to a Turaev-Viro-style construction of invariants of 3-manifolds equipped with links, surfaces or Seifert surfaces with their bounding link, again using a directed space structure, and using suitable bicategories suitably fibered over a small category encoding the adjacencies of (dimensions of) strata.  As \cite{DPY, LY1,Y3} the data associated with a defect will be richer than simply a module category over the spherical category associated with the bulk.  We refer to this circumstance as ``defects-with-structure.''

\section{Spherical Fusion Categories and  Turaev-Viro Invariants}


Let \textbf{k} be an algebraically closed field {\bf k}, and $\mathcal{A}$ be a spherical fusion category over \textbf{k}. See \cite{BW, EGNO} for complete definitions. We recall explicitly only enough of the structure to fix our notation and inform our generalization.  In particular $\mathcal{A}$ is semisimple with finitely many isomorphism classes of simple objects. We will choose as set $\mathcal{S}$ of simple objects representing the isomorphism clases, which without loss of generality contains the monoidal identity object and necessarily contains dual objects for each of its elements. 

We let $\textrm{dim} X \in \textrm{\textbf{k}} $
denotes the categorical dimension of the object $X$, which is the scalar multiple of the identity map on the (simple) monoidal identity object given by the right trace (or equivalently by sphericality the left trace) of its identity arrow, represented in string diagrams by an oriented loop labeled with the object on its descending portion (and its dual on the ascending portion).

In any semisimple spherical category, there is an analogue of Schur's Lemma:  namely any endomorphism $f:k\rightarrow k$ of a simple object $k$ is the identity morphism of that object multiplied by a scalar, namely ${\rm tr}(f)/{\rm dim}(k)$.  This has the consequence in manipulating string diagrams that if a string diagram with an edge labeled $k$ is multiplied by the trace of an endomorphism the same object $k$, $f:k\rightarrow k$, the result is the same as if $f$ were inserted on the string labeled $k$ and the resulting diagram instead multiplied by ${\rm dim}(k)$ or a loop labeled $k$.

We also choose bases for the hom-spaces ${\mathcal A}(i\otimes j,k)$ for all $i,j,k \in \mathcal S$, whose union we denote by $\mathcal B$.  It is easy to see by simplicity that there are dual bases  for ${\mathcal A}(k, i\otimes j)$ in the sense that the endomorphisms of $k$ resulting from composing elements of the chosen basis and the dual basis will be the identity arrow on $k$ or zero, and that the original bases give the projections of direct sum decompositions of the $i\otimes j$, for which the dual bases are the inclusions.   It is useful in conjunction with the analogue of Schur's Lemma to replace the dual bases with bases whose elements in ${\mathcal A}(k, i\otimes j)$ are ${\rm dim}(i)^{-1}$ times those in the dual basis.  We denote the unions of these rescaled dual bases by $\bar{\mathcal B}$.  In string diagrams, we will denote the scaled dual basis element by the same symbol as that of the corresponding element in $\mathcal B$, but labeling a node with one input and two outputs, rather than two inputs and one output.

The behavior of these bases when composed in the other order is then summarized in string diagrams by

\[\sum_{\begin{array}{c}k\in {\mathcal S}\\ \alpha \in {\mathcal B}\end{array}} {\rm dim}(k) \begin{array}{cc}

\begin{tikzpicture}[thick, scale=1.1]
    \node[circle, fill, inner sep=0pt, outer sep=0pt] (1) at (0,0){};
	\node[circle, fill, inner sep=0pt, outer sep=0pt] (2) at (1,0){};
	\node[circle, fill, inner sep=0pt, outer sep=0pt] (3) at (0.5,-1){};
	\node[circle, fill, inner sep=0pt, outer sep=0pt] (4) at (0.5,-2){};
	\node[circle, fill, inner sep=0pt, outer sep=0pt] (5) at (0,-3){};
	\node[circle, fill, inner sep=0pt, outer sep=0pt] (6) at (1,-3){};
	\node[circle, fill, inner sep=0pt, outer sep=0pt] (7) at (1.6,-1.54){};
	\node[circle, fill, inner sep=0pt, outer sep=0pt] (8) at (1.78,-1.54){};
	\node[circle, fill, inner sep=0pt, outer sep=0pt] (9) at (1.6,-1.62){};
	\node[circle, fill, inner sep=0pt, outer sep=0pt] (10) at (1.78,-1.62){};
	
	\draw [rounded corners] (1) -- (3) node[font=\tiny, midway, left]{$i$};
	\draw [rounded corners] (2) -- (3) node[font=\tiny, midway, right]{$j$};
	\draw [rounded corners] (3) -- (4) node[font=\small, midway, left]{$k$};
	\draw [rounded corners] (4) -- (5) node[font=\tiny, midway, left]{$i$};
	\draw [rounded corners] (4) -- (6) node[font=\tiny, midway, right]{$j$};
	\draw [rounded corners] (7) -- (8);
	\draw [rounded corners] (9) -- (10);
	
	\fill[opacity=.7] (3) circle (3pt) node[black, font=\small, right]{\,$\alpha$};
	\fill[opacity=.7] (4) circle (3pt) node[black,font=\small, right]{$\,\bar{\alpha}$};
\end{tikzpicture}

&
\begin{tikzpicture}[thick, scale=1]
    \node[circle, fill, inner sep=0pt, outer sep=0pt] (1) at (0,0){};
	\node[circle, fill, inner sep=0pt, outer sep=0pt] (2) at (1,0){};
	\node[circle, fill, inner sep=0pt, outer sep=0pt] (5) at (0,-3){};
	\node[circle, fill, inner sep=0pt, outer sep=0pt] (6) at (1,-3){};
	
	\draw [rounded corners] (1) -- (5) node[font=\small, midway, left]{$i$};
	\draw [rounded corners] (2) -- (6) node[font=\small, midway, right]{$j$};

\end{tikzpicture}
\end{array}\]

\noindent When we use this relation to manipulate sums of string diagrams in the direction which results in the identity morphism on the tensor product depicted on the right, we will say we are ``contracting the edge labeled $k$''.


\section{Biparcels}

Our goal in this section is to define the appropriate categorical structure which provides the appropriate common generalization of the spherical fusion categories of \cite{BW} and the parcels of \cite{LY1, LY2} to allow the description of Turaev-Viro type TQFTs with defects-with-structure.

The parcels of \cite{LY1, LY2} are both subsumed in the following definition:

\begin{definition} For a small gaunt category $\Gamma$, a $\Gamma$-parcel is a small category $\mathcal C$ equipped with a conservative functor $\gamma: {\mathcal C}\rightarrow \Gamma$.  A parcel is {\em fiber-finite} if the inverse image of any arrow in $\Gamma$ is finite, and {\em regular} if it is injective on the set of objects.
\end{definition}

In \cite{LY1} fiber-finite regular $P$-parcels for $P$ the totally ordered set with two or three elements were used as colors for a Dijkgraaf-Witten style construction of invariants of 3-manifolds with an embedded link and Seiftert surface to the link, corresponding to a endowing the triple $L \subset \Sigma \subset M$ as a very simple stratified space, with directed space structure \cite{Gr1, Gr2} in which directed paths may leave, but not enter, lower dimensional strata. \footnote{The authors recently learned that the idea used in \cite{DPY, LY1} of restricting admissible paths so that paths may not enter lower-dimensional strata had been used in 2009 by Treumann \cite{T} in the context of constructible stacks on stratified spaces.} In \cite{LY2}, parcels over the path category of the digraph with two vertices and two oppositely oriented edges joining them, were used as colors for a Dijkgraaf-Witten style construction of invariants of an oriented 3-manifold with an embedded oriented surface, corresponding to the directed space structure in which directed paths must enter or leave the surface in such a way that the orientation of the 3-manifold induced by appending a normal vector in the forward direction along the path to a basis of tangent vectors representing  the orientation on the surface agrees with the ambient orientation.

It is, of course, possible to see Wakui's original construction \cite{Wa} of Dijkgraaf-Witten theory as an instance of Barrett and Westbury's \cite{BW} generalized Turaev-Viro construction, by replacing the finite group $G$ with the spherical fusion category of $G$-graded finite dimensional \textbf{k}-vector spaces, with its associator given by the $\textbf{k}^\times$-valued 3-cocycle, $\alpha$.  Applying a similar transformation to the parcels of \cite{LY1, LY2}, equipped with a cocycle (though not a partial cocycle) will result in a \textbf{k}-linear bicategory with finite semisimple hom-categories, fibered over the 2-category arising from the gaunt base category by simply adjoining identity 2-arrows to every arrow.  Doing this suggests a more general notion which we will employ to extend the generalized Turaev-Viro construction to various types of directed spaces arising from manifolds stratified by submanifolds.

\begin{definition}
A {\em semisimple bicategory over} {\bf k} is a bicategory whose hom-categories are semisimple {\bf k}-linear abelian categories, in which both 1- and 2-dimensional composition of 2-arrows is \textbf{k}-bilinear.   Recall that a  semisimple \textbf{k}-linear category is called finite semisimple if there are only finitely many simple objects.  A semisimple bicategory is {\em finite} if each hom-category is finite semisimple.
\end{definition}

\begin{definition}
A {\em bicategory of fusion type over} {\bf k} is a finite semisimple bicategory over {\bf k} in which the identity 1-arrows of every object are simple as objects of the endomorphism category.
\end{definition}

Examples include fusion categories over \textbf{k} with their objects regarded as endo-1-arrows of a single object and their arrows regarded as 2-arrows.  Indeed the point of the definition is to have a many-objects analogue of a fusion category regarded as a one object bicategory.

\begin{definition}
A bicategory {\em has pivotal duals} if it is equipped with an assignment to each 1-arrow, $a$ of a 1-arrow $\bar{a}$ satisfying $\bar{\bar{a}} = a$, $s(\bar{a}) = t(a)$ and $t(\bar{a}) = s(a)$ for all 1-arrows $a$ and $\bar{\bf 1}_{X} = {\bf 1}_{X}$ for all objects $X$, together with 2-arrows $\epsilon_a: a\bar{a} \Rightarrow {\bf 1}_{s(a)}$ and $\eta_a: {\bf 1}_{t(a)} \Rightarrow \bar{a}a.$

For which the composite 2-arrows
\[ a \stackrel{\rho ^{-1}}{\Longrightarrow } a {\bf 1}_{t(a)}
\stackrel{a \eta_a}{\Longrightarrow } a(\bar{a} a)
\stackrel{\alpha ^{-1}_{a,\bar{a},a}}{\Longrightarrow }
	(a \bar{a} )a
\stackrel{\epsilon_a a}{\Longrightarrow } {\bf a}_{s(a)}a
\stackrel{\lambda }{\Longrightarrow } a \]

\noindent and
\[ \bar{a} \stackrel{\lambda ^{-1}}{\Longrightarrow } {\bf 1}_{t(a)} \bar{a}
\stackrel{\eta_a \bar{a}}{\Longrightarrow }
(\bar{a} a) \bar{a} \stackrel{\alpha_{\bar{a},a,\bar{a}}}{\Longrightarrow }
	\bar{a} (a \bar{a} )
\stackrel{\bar{a} \epsilon_a }{\Longrightarrow } \bar{a} {\bf 1}_{s(a)}
\stackrel{\rho }{\Longrightarrow } \bar{a} \]

\noindent are each identity 2-arrows.
\end{definition}

Observe that $\epsilon_a$ and $\eta_a$ make $\bar{a}$ into a right dual 1-arrow for $a$, while 
$\epsilon_{\bar{a}}$ and $\eta_{\bar{a}}$ make $\bar{a}$ into a left dual 1-arrow for $a$.


Now observe that in any bicategory with pivotal duals, every 1-arrow has canonically associated to it endo-2-arrows of the identity on its source and the identity on its target, $dim_s(a) = \eta_{\bar{a}} \epsilon_a$ and $dim_t(a) = \eta_a \epsilon_{\bar{a}}$, generalizing the right and left dimensions in a pivotal category.  And, as in the 1-object case, there are corresponding traces valued in the endo-2-arrows of the identities on the source and target of the 1-arrow given for any endo-2-arrow $f:a \Rightarrow a$ by
\[ tr_s(f) = \eta_{\bar{a}} \bar{a}f \epsilon_a \]
\[ tr_t(f) = \eta_a f\bar{a} \epsilon_{\bar{a}} \]

Now, at this level of generality, there is simply no analogue of the spherical condition that the right and left traces be equal:  the source and target traces are valued in different endo-hom-categories.  However, if the bicategory is of fusion type, or more generally has simple identity 1-arrows, each trace is an endo-2-arrow of a simple identity 1-arrow, and thus a scalar multiple of its identity 2-arrow, and we can thus make the following definition:

\begin{definition}
A {\em bicategory with spherical duals} is a bicategory with pivotal duals, in which all identity 1-arrows are simple, and for all 1-arrows $a$, for all endo-2-arrows $f:a \Rightarrow a$, if $tr_s(f) = c_s(f){\bf 1}_{s(a)}$ and $tr_t(f) = c_t(f){\bf 1}_{t(a)}$, the scalars $c_s(f)$ and $c_t(f)$ are equal.  We refer to this common scalar as the {\em scalar trace} of the endo-2-arrow, denoting it by ${\rm tr}(f)$, and the special case where $f = 1_X$ by ${\rm dim}(X)$.
\end{definition}

M\"{u}ger \cite{Mue} has called bicategories with spherical duals, ``spherical bicategories''.  We, however, eschew this terminology, since it more naturally would name a weak version of the spherical 2-categories of Mackaay \cite{Ma}.

In what follows if $\mathcal C$ is any category, we will denote by ${\mathcal C}_2$ its trivial 2-categorification, that is the 2-category created by adjoining a formal 2-identity arrow to each 1-arrow of the category, with the obvious 1- and 2-dimensional compositions.

It is also well established that string diagrams provide adequate computational representation of 1-arrows (as strings) and 2-arrows (as nodes) in a bicategory, if the regions of the diagram are colored with the objects of the bicategory, with the source to the left of the right.  In our calculations below, we omit the colors of the regions, as these can be recovered from the source/target date implied by the numbering on the vertices of the triangulation to which the string diagram has been associated.

\begin{definition} For a small gaunt category $\Gamma$, a {\em $\Gamma$-biparcel over \textbf{k}} is a semisimple bicategory over \textbf{k}, equipped with a 2-functor
$\beta:{\mathcal B} \rightarrow \Gamma_2$.

A $\Gamma$-biparcel over \textbf{k} is {\em finite} if the preimage of each 1-arrow together with its identity 2-arrow is a finite semisimple additive category (of 1- and 2-arrows) and {\em of fusion type} if it is finite and, moreover, the preimage of every identity 1-arrow (together with its identity 2-arrow) is a fusion category (with 1-composition as tensor product and 2-composition as composition).  A $\Gamma$-biparcel over \textbf{k} of fusion type is {\em spherical} if $\mathcal B$ has spherical duals, and {\em quasi-spherical} if $\mathcal B$ is a sub-bicategory of a bicategory with spherical duals.
\end{definition}

 Note that for general $\Gamma$, the condition that a $\Gamma$-biparcel be of fusion type is weaker than requring the bicategory $\mathcal B$ be of fusion type as a bicategory:  the finiteness condition is imposed not on whole hom-categories in $\mathcal B$, but only on the pre-images of 1-arrows of $\Gamma$.  For the finite posets used in \cite{LY1}, of course, the condition that a biparcel be of fusion type is equivalent to the bicategory $\mathcal B$ being of fusion type.

In subsequent sections of the paper we will show that quasi-spherical biparcels of fusion type provide an appropriate notion of colorings for a common generalization of our prior constructions \cite{LY1, LY2} and Barrett and Westbury's generalized Turaev-Viro theory \cite{BW}.   In the remainder of this section, we give several ways of contructing examples of biparcels from categorical structures more closely related to the usual toolkit for Turaev-Viro type TQFTs.  These require a little bit of preparation:

Many of our constructions will involve direct sums of abelian categories.  It is convenient to introduce a notation for an object concentrated in one direct summand.  Let $\Gamma_{[\delta]}$ denote the object of $\oplus_{d \in D} C_d$, $(C_d)_{d\in D}$ for which $C_\delta = \Gamma$ and $C_d = 0$ for all $d \neq \delta$.


\begin{proposition} If $({\mathcal C}, I, \otimes, \ldots)$ is a tensor (resp. semisimple tensor, pivotal, spherical) category over $\textbf{k}$, and $\mathcal G$ is a groupoid, there is a $\textbf{k}$-linear bicategory (resp. semisimple bicategory, bicategory of with pivotal duals, bicategory with spherical duals) ${\mathcal C} \sharp {\mathcal G}$ with the same objects as $\mathcal G$ and hom-categories given by

\[ {\mathcal C} \sharp {\mathcal G}(X,Y) = \bigoplus_{f \in {\mathcal G}(X,Y)} {\mathcal C} \]

\noindent with 1-dimensional composition given by
\[ (C_f)_{f \in {\mathcal G}(X,Y)},  (D_g)_{g \in {\mathcal G}(Y,Z)} \mapsto (\oplus_{fg = h} C_f\otimes D_g)_{h \in {\mathcal G}(X.Z)} \]

\noindent and 2-dimensional composition given coordinate-wise by the composition in ${\mathcal C}$ in each entry.

If moreover, $\mathcal C$ is a fusion category and the hom-sets of $\mathcal G$ are all finite, then ${\mathcal C} \sharp {\mathcal G}$ is of fusion type.

\end{proposition}

\begin{proof}
The identity 1-arrow for an object $X$ is $I_{[1_X]} \in {\mathcal C} \sharp {\mathcal G}(X,X)$, where $I$ is the monoidal identity of $\mathcal C$.

The associators and unitors for the bicategory structure are induced by those of $\mathcal C$ together with the distributivity coherence maps that exist by the exactness of $\otimes$ in both variables in what by now is a fairly routine way.  

That the resulting bicategory is semisimple when $\mathcal C$ is immediate as a direct sum of semisimple abelian categories is semisimple.

If $\mathcal C$ is pivotal, with the dual of $C$ denoted $\overline{C}$, then the dual of $(C_f)_{f \in {\mathcal G}(X,Y)}$ is $(\Gamma_g)_{g \in {\mathcal G}(Y,X)}$ where $\Gamma_g = \overline{C_{g^{-1}}}$.   To describe the structure maps for the duality, it suffices to describe them for 1-arrows concentrated in a direct summand of some ${\mathcal C} \sharp {\mathcal G}(X,Y)$, but here they are simply the structure maps for the duality in $\mathcal C$. Thus if $\mathcal C$ is spherical, the necessary equality for the pivotal duals in ${\mathcal C}\sharp {\mathcal G}$ to be spherical duals follows immediately from the condition in $\mathcal C$.

Finally, the last statement of the proposition follows trivially, once it is noted that the identity 1-arrow on an object $X$ is $I_{[{Id_X}]}$.
\end{proof}

Recall from Etingoff et al. \cite{EGNO} the definition of a grading of a tensor category:

\begin{definition} For a group $G$ a {\em $G$-grading} of a tensor category $\mathcal C$ is an isomorphism of categories
\[ {\mathcal C} \cong \bigoplus_{g \in G} {\mathcal C}_g \]

\noindent where for each $g \in G$, ${\mathcal C}_g$ is an abelian subcategory of $\mathcal G$ such that $X \in {\mathcal C}_g$ and $Y \in {\mathcal C}_h$ implies $X\otimes Y \in {\mathcal C}_{gh}$.
\end{definition}

For our purposes it will be useful to think this in a slightly different way:

\begin{definition} A {\em $G$-grading span} for a tensor category $\mathcal C$ is a span of bicategory functors
\[ G_2 \stackrel{| \; |}{\longleftarrow} {\mathcal C}_{hom} \stackrel{\iota} {\longrightarrow}{\mathcal C} \]

\noindent where ${\mathcal C}_{hom}$ is a monoidal category with $\iota$ one functor of a monoidal equivalence of categories with its image in $\mathcal C$, which image is a monoidal full subcategory of $\mathcal C$, whose closure under direct sum is all of $\mathcal C$,  and the preimage of each 1-arrow in $G_2$ together with its identity 2-arrow is an abelian category. 
\end{definition}

In what follows both $\mathcal C$ and ${\mathcal C}_{hom}$ will be regarded as bicategories with a single object.

The spirit of the definition is that ${\mathcal C}_{hom}$ should be the subcategory of $\mathcal C$ of objects homogeneous with respect to the grading and $\iota$ the inclusion functor, but the fussiness of the definition is needed since each fiber of the degree functor $| \; |$ must have a zero object, and  $\mathcal C$ need not have as many distinct zero objects as there are elements in $G$.

It is easy to see that we have

\begin{proposition}
A $G$-grading of a tensor category $\mathcal C$ is equivalent to a $G$-grading span for $\mathcal C$.
\end{proposition}

\begin{proof}
Given a $G$-grading define ${\mathcal C}_{hom}$ to be
\[ \coprod_{g \in G} {\mathcal C}_g  . \]

The monoidal product on $\mathcal C$ induces a monoidal product on ${\mathcal C}_{hom}$ given by the restriction of the monoidal product to the grades followed by the inclusion of the grade indexed by the product of the degrees into the coproduct, and $\iota$ is the canonical functor induced by the inclusions of the grades into $\mathcal C$, which is trivially seen to be a monoidal functor.  The image is precisely the homogeneous objects of $\mathcal C$ with respect to the grading, and thus its closure under direct sum is all of $\mathcal C$.

Plainly there is a bicategory functor from ${\mathcal C}_{hom}$ to $G_2$ sending the 1-arrows to their degree and every 2-arrow to the identity 2-arrow of the common degree of its source and target.  That this assignment respects 1-dimensional composition follows from the condition in a grading that $X \in C_g$ and $Y \in C_h$ implies $X\otimes Y \in C_{gh}$.

Conversely given a grading span for a tensor category $\mathcal C$, for each element of $G$, define ${\mathcal C}_g$ to be the image under $\iota$ of the preimage of $g$ as a 1-arrow in $G_2$ together with its identity 2-arrow.  The generation condition on the image of $\iota$ implies that ${\mathcal C} \cong \oplus_{g \in G} {\mathcal C}_g$, while the functoriality of the degree functor implies that $X \in C_g$ and $Y \in C_h$ implies $X\otimes Y \in C_{gh}$.
\end{proof}

This last result gives us a construction for parcels.

First, we consider the effect of pulling back the degree functor along the 2-functor induced by a functor $\Phi:{\mathcal G}\rightarrow G$, for $\mathcal G$ a groupoid.

\begin{proposition}  For $\mathcal G$ a groupoid, given a grading span $G_2 \stackrel{| \; |}{\leftarrow} {\mathcal C}_{hom} \stackrel{\iota} {\rightarrow}{\mathcal C} $ for a tensor (resp.  pivotal, spherical) category over ${\bf k}$, then for any functor $\Phi:{\mathcal G} \rightarrow G$, the pullback of $| \; |:{\mathcal C}_{hom}\rightarrow G_2$ along the induced 2-functor $\Phi_2:{\mathcal G}_2\rightarrow G_2$ is a bicategory over ${\bf k}$ (resp. bicategory with pivotal duals, bicategory with spherical duals).  Moreover, if $\mathcal C$ is a fusion category, the preimage of every 1-arrow in ${\mathcal G}_2$ together with its identity 2-arrow will be finite semisimple, and thus if every hom-set of $\mathcal G$ is finite, the bicategory will be of fusion type.
\end{proposition}

\begin{proof}
It is trivial that the pullback of ${\mathcal C}_{hom}$ is a bicategory over ${\bf k}$ when $\mathcal C$ is a tensor category over ${\bf k}$.  Now suppose $\mathcal C$ is pivotal.  It follows that ${\mathcal C}_{hom}$ is pivotal as the dual of any object homogeneous of degree $g$ is an object homogenenous of degree $g^{-1}$, and thus, ${\mathcal C}_{hom}$ is closed under taking duals.  Likewise the spherical condition on traces is preserved under restriction to ${\mathcal C}_{hom}$.  Passing to the pullback, 1-arrows in the pullback are pairs $(X,\gamma)$ where $X$ is an object of degree $g$ and $\Phi(\gamma) = g$. It is easy to see that the structure maps for $X$ and $\overline{X}$ $\epsilon_X$ and $\eta_X$ give rise to structure maps $(\epsilon_X, 1)$ and $(\eta_X, 1)$ making
$(\overline{X}, \gamma^{-1})$ into the dual 1-arrow of $(X,\gamma)$. 

 In the spherical case, since the identity 1-arrows of any object $(*,S)$ of the pullback are $(I,1_S)$ the simplicity of $I$ implies the simplicity of all identity 1-arrows, the equality of the left and right trace in $\mathcal C$ implies the equality of the scalars when the source and target traces of $(f,1_\gamma)$ are written as scalar multiples of $1_{(I,s(\gamma))}$ and $1_{(I,t(\gamma))}$, respectively.

Finally, if $\mathcal C$ is fusion, note that the preimage of each $\gamma$ in $\mathcal G$ is a copy of $C_{\Phi(g)}$, and thus finite semisimple, from which the statement with the finiteness condition on $\mathcal G$ follows immediately.
\end{proof}

\begin{proposition}   If $\Gamma$ is a gaunt category which embeds in ${\mathcal G}(\Gamma)$, the groupoid  it freely generates, then given a grading span $G_2 \stackrel{| \; |}{\leftarrow} {\mathcal C}_{hom} \stackrel{\iota} {\rightarrow}{\mathcal C} $ for a tensor (resp.  spherical, fusion) category, then for any functor $\Phi:\Gamma \rightarrow G$, the pullback of $| \; |:{\mathcal C}_{hom}\rightarrow G_2$ along the induced 2-functor $\Phi_2:\Gamma_2\rightarrow G_2$ is a $\Gamma$-biparcel (resp. quasi-spherical  $\Gamma$-biparcel, $\Gamma$-biparcel of fusion type).
\end{proposition}

\begin{proof}
This is immediate from the previous proposition and the definitions of the various types of biparcels.
\end{proof}

Similarly biparcels can be constructed by pulling back sectors of a multifusion category.

\begin{proposition} A multifusion category $({\mathcal C}, I, \otimes, \ldots)$ over $k$ in which $I$ has direct summands $I_j$ for $j\in J$ for some finite indexing set $J$ is equivalent to a bicategory of fusion type over $k$ with objects set $J$. Moreover, $\mathcal C$ is spherical in the sense of Cui and Wang \cite{CW} if and only if the corresponding bicategory has spherical duals.
\end{proposition}

\begin{proof}
Given a multifusion category $\mathcal C$ with $I = \oplus_{j\in J} I_j$, define a bicategory $\hat{\mathcal C}$ with object set $J$ by letting $\hat{\mathcal C} (i,j) = {\mathcal C}_{i,j}$, the $i,j$-sector of $\mathcal C$, and defining the 1-dimensional composition to be the restriction of tensor product to the sectors.

Conversely given a bicategory of fusion type $\mathcal B$ with object set $J$, the corresponding multifusion category is 

\[ \check{\mathcal B} = \oplus_{i,j \in J} {\mathcal B}(i,j) \]

\noindent with tensor product defined by extending the 1-dimensional composition by the universal property of the direct sum.

It is easy to see that dual objects in the multifusion category induce dual 1-arrows in the corresponding bicategory and conversely, and that the spherical condition is preserved by the construction in both directions.
\end{proof}

Now observe that the bicategory $\hat{\mathcal C}$ is naturally equiped with a functor to the trivial 2-categorification of the chaotic preorder on $J$, $Ch(J)_2$.

It is then easy to see that the following holds.

\begin{proposition}
For $\mathcal C$ a multifusion category with $I = \oplus_{j\in J} I_j$, if $\Gamma$ is a guant category and $\Phi:\Gamma\rightarrow Ch(J)$ is a functor to the chaotic preorder on $J$, the pullback of $\hat{\mathcal C}$ along $\Phi_2$ with its functor to $\Gamma_2$ is a $\Gamma$-biparcel of fusion type.  If, moreover, $\Gamma$ embeds in the groupoid it freely generates, and $\mathcal C$ is a spherical multifusion category, the pullback is a quasi-spherical $\Gamma$-biparcel.
\end{proposition}

\section{ Manifolds with Defects, Flag-Like Triangulations and Directed Space Structures}

In \cite{LY1,LY2} the authors considered various configurations of a 3-manifold equipped with specified submanifolds, seeing them as very simple instances of stratified spaces, and using the stratification to equip them in some way with a directed space structure in the sense of Grandis \cite{Gr1, Gr2}.  We here recall from \cite{Gr1,Gr2,CY,LY1,LY2} for the present work.

As in \cite{LY1} we use the notions of Crane and Yetter \cite{CY}

\begin{definition} A {\em starkly stratified space} is a PL space $X$ equipped with a filtration

\[ X_0 \subset X_1 \subset \ldots \subset X_{n-1} \subset X_n = X \]

\noindent satisfying

\begin{enumerate}
\item There is a triangulation $\mathcal T$ of $X$ in which each $X_k$ is a subcomplex.
\item For each $k = 1,\ldots n$ $X_k \setminus X_{k-1}$ is a(n open) $k$-manifold.
\item If $C$ is a connected component of $X_k \setminus X_{k-1}$, then $\mathcal T$ restricted
to $\bar{C}$ gives $\bar{C}$ the structure of a combinatorial manifold with boundary.
\item For each combinatorial ball $B_k$ with $\accentset{\circ}B_k \subset X_k \setminus X_{k-1}$, $\accentset{\circ}B_k$ admits a closed neighborhood $N$ given inductively as a cell complex as follows (although we require $B_k$ to be a combinatorial ball, the triangulation is then ignored):

\noindent $N = N_n$, where $N_m$ for $k \leq m \leq n$ is given inductively by 

\[ N_k = B_k\] 

\noindent and 

\[ N_{\ell+1} = N_\ell \cup \bigcup_{v \in S_\ell} L(v) \ast v \]

\noindent for $S_\ell$ a finite set of points in $X\setminus X_\ell$, andl $L(\cdot)$ a function on $S_\ell$
valued in 

\[ \{ L | L\; \mbox{\rm is a combinatorial ball and}\; B_k \subset L \subset N_\ell  \}\]

\noindent We will call such a neighborhood of the interior of a combinatorial ball lying in a stratum of the same dimension a {\em stark neighborhood}.

\end{enumerate}
\end{definition}

As we observed in \cite{LY1} for a knot or link $K$ with a Seifert surface $\Sigma$ in a 3-manifold $M$, the filtered space $\emptyset \subset K \subset \Sigma \subset M$ is a starkly stratified space.  And, as in \cite{LY1}, even though $K$ may, in fact, be a link, we will always refer to it as ``the knot $K$'' to avoid confusion with the other meaning of the word ``link'' in PL topology.

Recall also

\begin{definition}

A {\em simplicial flag} is a finite sequence of simplexes 

\[ \sigma_0 \subset \sigma _1 \subset \ldots \subset \sigma_n \]

\noindent such that each $\sigma_i$ is a face of $\sigma_{i+1}$.  A simplicial flag is {\em complete} if $\sigma_i$ is $i$-dimensional, and thus the sequence may be formed by starting with an $n$-simplex $\sigma_n$ and iteratively chosing a codimension 1 face until a vertex is reached.
\end{definition}

\begin{definition}
A triangulation $\mathcal T$ of a stratified PL space 

\[ X_0 \subset X_1 \subset \ldots \subset X_{n-1} \subset X_n = X \]

\noindent is {\em flag-like} if each of the $X_i$ is a subcomplex, and moreover for each simplex $\sigma$ of
$\mathcal T$, the restriction of the filtration to the simplex, that is the distinct non-empty intersections in the sequence

\[ X_0\cap \sigma \subset X_1 \cap \sigma \subset \ldots \subset X_{n-1} \cap \sigma \subset \sigma \]

form a (possibly incomplete) simplicial flag.

\end{definition}

As observed in \cite{LY1}, any triangulation of a stratified space for which the closures of the strata are subcomplexes, gives rise to a flag-like triangulation upon barycentric subdivision.

As was shown in \cite{LY1}, two flag-like triangulations of a starkly stratified space are combinatorially equivalent if and only if they are related by a sequence of extended Pachner moves (extended bistellar flips) each of which is flag-like in the sense that both its initial and final state are flag-like.  Moreover, in the case of a knot, Seifert surface, 3-manifold triple, a sufficient set of combinatorial moves remains if the 4-8 (resp. 8-4) extended bistellar move in which and edge of the knot is subdivided (resp. its reverse in which two edges on the knot are welded) is replaced with 3-6 (resp. 6-3) move in which an edge (resp. pair of edges) on the knot with a three-edge link, with a single vertex in the Seifert surface is subdivided (resp. welded).

Building on the counting interpretation provided for some of the invariants constructed in \cite{LY1} in terms of Grandis's fundamental category \cite{Gr1}, in \cite{LY2} the authors considered various constructions of directed space structures from stratifications.  In particular we recall from Grandis \cite{Gr1}

\begin{definition}  A {\em directed topological space}  or {\em $d$-space} $X = (X,dX)$ is a topological space $X$, equipped with a set $dX$ of continuous maps $\phi:{\bf I}\rightarrow X$ satisfying

\begin{enumerate}
\item Every constant map $x:{\bf I}\rightarrow X$ for $x \in X$ is in $dX$.
\item $dX$ is closed under pre-composition with continuous weakly monotone functions from ${\bf I}$ to ${\bf I}$.
\item $dX$ is closed under concatenation of paths.
\end{enumerate}
\bigskip

\noindent Elements of $dX$ are called {\em directed paths} in $X$.  As was observed in \cite{LY2} Grandis's closure condition is equivalent to $dX$ being closed under both (weakly monotone) reparametrizations of paths and arbitrary factorizations of paths with respect to concatenation.

A {\em map of directed spaces} $f:(X,dX)\rightarrow (Y,dY)$ is a continuous function $f:X\rightarrow Y$ such that $p \in dX$ implies $f(p) \in dY$.
\end{definition}

In \cite{LY1} a directed space structure was constructed from the knot, Seifert surface, 3-manifold triple by defining a path to be directed if the function assigning to each point the dimension of the stratum in which it lies was (weakly) monotone increasing on the path. While in \cite{LY2} other ways of producing a directed space structure were considered, in particular that dual to the one in \cite{LY1} and one in which for a pair of an oriented manifold with an oriented submanifold of codimension one, paths were directed if they admitted a factorization into paths which either lay within a stratum or were factors of paths transverse to the submanifold with local intersection number $+1$ at each intersection point.

In each case the fundamental category $\uparrow \Pi (X)$ admitted a conservative functor to a gaunt category with one object for each dimension of stratum, the poset $\mathsf 3$ or ${\mathcal P}(\Gamma_2)$ the path category on the directed graph with a pair of vertices and one edge in each direction between them.

Finally, we need to describe the appropriate relationship between the directed space structure induced from a stratification and flag-like triangulations of the stratified space.  As in \cite{LY2} the appropriate means of doing this is to order the vertices within each stratum, then orient edges lying within a stratum from earlier vertex to later vertex, while orienting edges that go between strata so that the edge traversed in the direction of the orientation is a directed path. As in \cite{LY2} we refer to a flag-like triangulation with such an orientation on its edges as a {\em directed triangulation}.

\section{ Biparcel Colorings and TQFTs with Defects-with-Structure}

In the circumstances described in the last section, in which a starkly stratified space has been equipped with a directed space structure for which all paths lying entirely within a stratum are directed, and the fundamental category of the directed space has been equipped with a conservative functor to a gaunt category with objects corresponding dimensions of the strata, we describe a procedure for producing state-sum invariants, which for reasons analogous to those discussed in \cite{Y2} will be multiplicative under disjoint union and admit TQFT-type factorizations along codimension one subspaces transverse to the strata.

\begin{definition}
For a quasi-spherical $\Gamma$-biparcel of fusion type $\Phi:{\mathcal C}\rightarrow \Gamma_2$, with $J:{\mathcal C}\rightarrow {\mathcal C}^\prime$ the inclusion of $\mathcal C$ into a bicategory with spherical duals, {\em a system of $\Phi$-coloring data } consists of a choice of representatives for the isomorphism classes of simple 1-arrows in the preimage of each 1-arrow of $\Gamma_2$, which we collect into a set of 1-arrows $S$ with $S$ chosen to be closed under taking dual 1-arrows which lie in $\mathcal C$, together with a choice of basis for each of the spaces of 2-arrows ${\mathcal C}(ab,c)$ where $a, b, c \in S$, with $a$ and $b$ composable and $c$ parallel to their composition.  We denote the union of these bases by $B$, and the entire system of $\Phi$-coloring data by $(\Phi, S, B)$.
\end{definition}


\begin{definition} \label{nice circumstance} 
Let $(X,dX)$ is the directed space associated to a 3-manifold $X$, filtered by the inclusion of a knot, a (closed) surface, or a knot and Seifert surface to the knot, $\delta:\uparrow \Pi(X,dX)\rightarrow \Gamma$ a functor from the fundamental category to a gaunt category $\Gamma$, with objects the dimensions of the strata, which maps each object (point) to the dimension of the stratum in which it lies. For $\Phi$ a quasi-spherical $\Gamma$-biparcel and $\Phi$-coloring data $(\Phi,S,B)$ a $(\Phi,S,B)$-coloring $\lambda$ of a directed triangulation $\mathcal T$ of $(X,dX)$ is an assignment to each vertex of $\mathcal T$ of the object naming the dimension of its stratum, to each edge $e$ of the triangulation, a 1-arrow $\lambda(e) \in S$ which lies over the 1-arrow of $\Gamma_2$ to which the edge, as a directed path, maps under $\delta$, and of a 2-arrow $\lambda(uvw) \in B$ with target $\lambda(uv)\lambda(vw)\bar{\lambda(uw)}$ to each 2-simplex $uvw$.  We denote the set of $(\Phi,S,B)$-colorings of $\mathcal T$ by $\Lambda_{(\Phi,S,B)}({\mathcal T})$ or simply $\Lambda({\mathcal T})$ if
the coloring data are clear from contex.
\end{definition}

As usual, the topological invariant will arise by summing over all colorings a quantity which is a product of contributions from each simplex of the triangulation (here required to be flag-like and directed).  The most important factor are those associated to colorings of 3-simplexes simplexes, which we denote $\alpha_{\pm}(\lambda(\sigma))$.  There will be no factors associated to 2-simplexes.  That associated to a 1-simplex will be the scalar dimension of the 1-arrow with which it is colored, and that associated to vertices will be $c_n^{-1}$ where $c_n$ is the sum over representatives of isomorphism classes of simple endo-1-arrows of the object $n$ of the squares of their dimensions, where $n$ is the dimension of the stratum in which the vertex lies.

Here, $\alpha_{\epsilon(\sigma)}(\lambda(\sigma))$ is the scalar trace of the endo-2-arrow given by the string diagram shown next to the 3-simplex, where the sign $\epsilon{\sigma}$ is positive or negative according to whether the combinatorial orientation of the 3-simplex agrees with or is opposite that of the 3-manifold (taken in the illustration to be the right-hand-rule orientation).

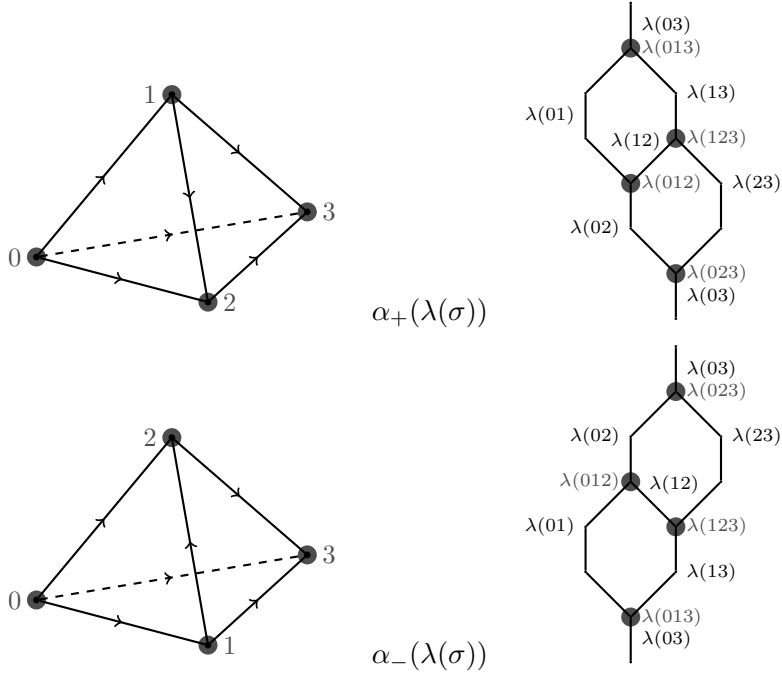
\begin{center}
\begin{tabular}{ccc}

\begin{tikzpicture}[thick, scale=1.2]
    \node[circle, fill, inner sep=.9pt, outer sep=0pt] (A) at (0,0){};
	\node[circle, fill, inner sep=.9pt, outer sep=0pt] (B) at (1.9,-0.5){};
	\node[circle, fill, inner sep=.9pt, outer sep=0pt] (C) at (3,0.5){};
	\node[circle, fill, inner sep=.9pt, outer sep=0pt] (D) at (1.5,1.8){};
	\draw [dashed,midarrow={>}] (A)--(C);	
	\draw [midarrow={>}] (A)--(B);
	\draw [midarrow={>}] (B)--(C);
	\draw [midarrow={>}] (D)--(C);
	\draw [midarrow={>}] (D)--(B);
	\draw [midarrow={>}] (A)--(D);
	\fill[opacity=.7] (A) circle (3pt) node[font=\small, left]{$0$\,};
	\fill[opacity=.7] (B) circle (3pt) node[font=\small, right]{\,$2$};
	\fill[opacity=.7] (C) circle (3pt) node[font=\small, right]{\,$3$};
	\fill[opacity=.7] (D) circle (3pt) node[font=\small, left]{$1$\,};
\end{tikzpicture}

&
 $\alpha_+(\lambda(\sigma))$
&

\begin{tikzpicture}[thick, scale=1.2]
    \node[circle, fill, inner sep=0pt, outer sep=0pt] (1) at (-0.5,3.5){};
	\node[circle, fill, inner sep=0pt, outer sep=0pt] (2) at (-0.5,3){};
	\node[circle, fill, inner sep=0pt, outer sep=0pt] (3) at (-1,2.5){};
	\node[circle, fill, inner sep=0pt, outer sep=0pt] (4) at (0,2.5){};
	\node[circle, fill, inner sep=0pt, outer sep=0pt] (5) at (-1,2){};
	\node[circle, fill, inner sep=0pt, outer sep=0pt] (6) at (0,2){};
	\node[circle, fill, inner sep=0pt, outer sep=0pt] (7) at (-0.5,1.5){};
	\node[circle, fill, inner sep=0pt, outer sep=0pt] (8) at (0.5,1.5){};
	\node[circle, fill, inner sep=0pt, outer sep=0pt] (9) at (-0.5,1){};
	\node[circle, fill, inner sep=0pt, outer sep=0pt] (10) at (0.5,1){};
	\node[circle, fill, inner sep=0pt, outer sep=0pt] (11) at (0,0.5){};
	\node[circle, fill, inner sep=0pt, outer sep=0pt] (12) at (0,0){};
	\draw [rounded corners] (1) -- (2) node[font=\tiny, midway, right]{$\lambda(03)$};
	\draw [rounded corners] (2) -- (3);
	\draw [rounded corners] (2) -- (4) node[font=\tiny, right]{$\lambda(13)$};
	\draw [rounded corners] (3) -- (5) node[font=\tiny, midway, left]{$\lambda(01)$};
	\draw [rounded corners] (4) -- (6);
	\draw [rounded corners] (5) -- (7);
	\draw [rounded corners] (6) -- (7) node[font=\tiny, midway, above]{$\lambda(12)$\,\,\,\,\,\,\,\,\,\,};
	\draw [rounded corners] (6) -- (8) node[font=\tiny, right]{$\lambda(23)$};
	\draw [rounded corners] (7) -- (9) node[font=\tiny, left]{$\lambda(02)$};
	\draw [rounded corners] (9) -- (11);
	\draw [rounded corners] (8) -- (10) -- (11);
	\draw [rounded corners] (11) -- (12) node[font=\tiny, midway, right]{$\lambda(03)$};
	\fill[opacity=.7] (2) circle (3pt) node[font=\tiny, right]{$\lambda(013)$};
	\fill[opacity=.7] (6) circle (3pt) node[font=\tiny, right]{$\lambda(123)$};
	\fill[opacity=.7] (7) circle (3pt) node[font=\tiny, right]{$\lambda(012)$};		
	\fill[opacity=.7] (11) circle (3pt) node[font=\tiny, right]{$\lambda(023)$};		
\end{tikzpicture}

\end{tabular}
\end{center}

\begin{center}
\begin{tabular}{ccc}

\begin{tikzpicture}[thick, scale=1.2]
    \node[circle, fill, inner sep=.9pt, outer sep=0pt] (A) at (0,0){};
	\node[circle, fill, inner sep=.9pt, outer sep=0pt] (B) at (1.9,-0.5){};
	\node[circle, fill, inner sep=.9pt, outer sep=0pt] (C) at (3,0.5){};
	\node[circle, fill, inner sep=.9pt, outer sep=0pt] (D) at (1.5,1.8){};
	\draw [dashed,midarrow={>}] (A)--(C);	
	\draw [midarrow={>}] (A)--(B);
	\draw [midarrow={>}] (B)--(C);
	\draw [midarrow={>}] (D)--(C);
	\draw [midarrow={>}] (B)--(D);
	\draw [midarrow={>}] (A)--(D);
	\fill[opacity=.7] (A) circle (3pt) node[font=\small, left]{$0$\,};
	\fill[opacity=.7] (B) circle (3pt) node[font=\small, right]{\,$1$};
	\fill[opacity=.7] (C) circle (3pt) node[font=\small, right]{\,$3$};
	\fill[opacity=.7] (D) circle (3pt) node[font=\small, left]{$2$\,};
\end{tikzpicture}

&
 $\alpha_-(\lambda(\sigma))$
&

\begin{tikzpicture}[thick, scale=1.2]
    \node[circle, fill, inner sep=0pt, outer sep=0pt] (1) at (0,0){};
	\node[circle, fill, inner sep=0pt, outer sep=0pt] (2) at (0,-0.5){};
	\node[circle, fill, inner sep=0pt, outer sep=0pt] (3) at (-0.5,-1){};
	\node[circle, fill, inner sep=0pt, outer sep=0pt] (4) at (0.5,-1){};
	\node[circle, fill, inner sep=0pt, outer sep=0pt] (5) at (-0.5,-1.5){};
	\node[circle, fill, inner sep=0pt, outer sep=0pt] (6) at (0.5,-1.5){};
	\node[circle, fill, inner sep=0pt, outer sep=0pt] (7) at (-1,-2){};
	\node[circle, fill, inner sep=0pt, outer sep=0pt] (8) at (0,-2){};
	\node[circle, fill, inner sep=0pt, outer sep=0pt] (9) at (-1,-2.5){};
	\node[circle, fill, inner sep=0pt, outer sep=0pt] (10) at (0,-2.5){};
	\node[circle, fill, inner sep=0pt, outer sep=0pt] (11) at (-0.5, -3){};
	\node[circle, fill, inner sep=0pt, outer sep=0pt] (12) at (-0.5,-3.5){};
	\draw [rounded corners] (1) -- (2) node[font=\tiny, midway, right]{$\lambda(03)$};
	\draw [rounded corners] (2) -- (3) node[font=\tiny, left]{$\lambda(02)$};
	\draw [rounded corners] (2) -- (4) node[font=\tiny, right]{$\lambda(23)$};
	\draw [rounded corners] (3) -- (5);
	\draw [rounded corners] (4) -- (6);
	\draw [rounded corners] (5) -- (7) node[font=\tiny, left]{$\lambda(01)$};
	\draw [rounded corners] (6) -- (8);
	\draw [rounded corners] (5) -- (8) node[font=\tiny, midway, above]{\,\,\,\,\,\,\,\,\,\,\,\,$\lambda(12)$};
	\draw [rounded corners] (7) -- (9);
	\draw [rounded corners] (9) -- (11);
	\draw [rounded corners] (8) -- (10) node[font=\tiny, right]{$\lambda(13)$};
	\draw [rounded corners] (10) -- (11);
	\draw [rounded corners] (11) -- (12) node[font=\tiny, midway, right]{$\lambda(03)$};
	\fill[opacity=.7] (2) circle (3pt) node[font=\tiny, right]{$\lambda(023)$};
	\fill[opacity=.7] (5) circle (3pt) node[font=\tiny, left]{$\lambda(012)$};
	\fill[opacity=.7] (8) circle (3pt) node[font=\tiny, right]{$\lambda(123)$};		
	\fill[opacity=.7] (11) circle (3pt) node[font=\tiny, right]{$\lambda(013)$};		
\end{tikzpicture}

\end{tabular}
\end{center}
\bigskip

Our main theorem is then 

\begin{theorem}
For any directed space $(X,dX)$ as in \ref{nice circumstance}, the quantity 

\[ \sum_{\lambda \in \Lambda({\mathcal T})} \prod_{v \in {\mathcal T}_0} c_{\lambda(v)}^{-1} \prod_{uv \in {\mathcal T}_1} {\rm dim}(\lambda(uv)) \prod_{\sigma \in {\mathcal T}_3} \alpha_{\epsilon(\sigma)}(\lambda(\sigma)) \]

\noindent is independent of the triangulation $\mathcal T$ and thus an invariant of the directed space $(X,dX)$.

\end{theorem}

\begin{proof}
The proof proceeds as usual, by verifying invariance under the extended Pachner moves.  The verification for the flag-like Pachner moves proper, regardless of how the boundary interacts with the stratification are essentially very similar to that provide below for the 2-6 and 3-6 moves, and entirely analogous to the proofs of invariance in \cite{BW, LY1, Y} and are left to the reader.  As usual for the reasons set forth in \cite{LY1, Y} invariance under these moves, also implies invariance under reordering of the vertices within the strata.  For brevity we abuse notation and neither close the string diagrams representing the instances of $\alpha_\pm$ to form their traces geometrically, nor write them inside the symbols ${\rm tr}(\;\;)$.  All of the string diagrams in the calculations below should be understood as the scalar traces of what is actually written. We also, as noted above, omit the colorings of regions by objects, as these can be recovered from the numbering conventions of the vertices which recurs in the colors on edges and vertices of the string diagrams.

For each of the moves we explicitly verify, we first indicate, with a depiction of the simplicial geometry, the orientations on the 3-simplexes, determining whether an instance of $\alpha_+$ or of $\alpha_-$ is the appropriate coefficient, and then give the algebraic local contribution of the simplexes depicted in the before and after states of the move.  The calculation performed then shows that the local contribution of both states is the same, thus establishing invariance under the move, once it is observed that the full state-sum for any pair of triangulations related by an instance of the move can be rearranged by distributivity and commutativity to have summands corresponding to each coloring of the boundary of the region depicted, such summand having the quantity computed as a factor when the region is triangulated as indicated in the depictions of the two states. 

\medskip

\begin{itemize}
    \item [2-6] move.  Here a subdivision of two 3-simplexes sharing a face which is a 2-simplex in a 2-dimensional stratum whose link is a pair of points is induced by applying an Alexander subdivision of the shared 2-simplex.  There are several cases according to the way in which the 2-dimensional stratum restricts directed paths.
    
    First, two tetrahedron states of 2-6 move are the following.

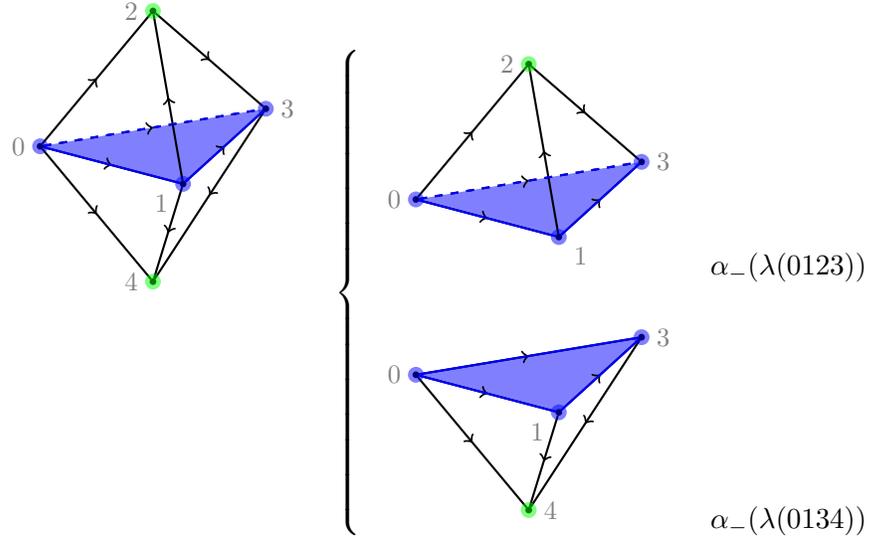
\begin{center}
\begin{tabular}{ll}

$\displaystyle 
\begin{tikzpicture}[thick][scale=0.7]
\fill[blue,blue, opacity=.5] (0,0)--(1.9,-0.5)--(3,0.5)--(0,0);
	\node[circle, fill, inner sep=.9pt, outer sep=0pt] (A) at (0,0){};
	\node[circle, fill, inner sep=.9pt, outer sep=0pt] (B) at (1.9,-0.5){};
	\node[circle, fill, inner sep=.9pt, outer sep=0pt] (C) at (3,0.5){};
	\node[circle, fill, inner sep=.9pt, outer sep=0pt] (D) at (1.5,1.8){};
	\node[circle, fill, inner sep=.9pt, outer sep=0pt] (E) at (1.5,-1.8){};
		\begin{scope}
			\draw [midarrow={>}] (A)--(B);
			\draw [midarrow={>}] (B)--(C);
			\draw [dashed,midarrow={>}] (A)--(C);	
			\draw [midarrow={>}] (A)--(D);
			\draw [midarrow={>}] (B)--(D);
			\draw [midarrow={>}] (D)--(C);
			\draw [midarrow={>}] (A)--(E);
			\draw [midarrow={>}] (B)--(E);
			\draw [midarrow={>}] (C)--(E);
			\fill[blue, opacity=.5] (A) circle (3pt) node[black, font=\small, left]{\,$0$\,};
			\fill[blue, opacity=.5] (B) circle (3pt) node[black, font=\small, below left]{\,$1$\,};
			\fill[blue, opacity=.5] (C) circle (3pt) node[black, font=\small, right]{\,$3$\,};
			\fill[green, opacity=.5] (D) circle (3pt) node[black, font=\small, left]{\,$2$\,};
			\fill[green, opacity=.5] (E) circle (3pt) node[black, font=\small, left]{\,$4$\,};
			\draw[thick,blue,blue] (A)--(B);
			\draw[thick,blue,blue] (B)--(C);
			\draw[dashed,thick,blue,blue] (A)--(C);			
		\end{scope}
\end{tikzpicture}$

&
 
$ \left\{ 
 \begin{array}{ll}

\begin{tikzpicture}[thick][scale=0.7]
\fill[blue,blue, opacity=.5] (0,0)--(1.9,-0.5)--(3,0.5)--(0,0);
	\node[circle, fill, inner sep=.9pt, outer sep=0pt] (A) at (0,0){};
	\node[circle, fill, inner sep=.9pt, outer sep=0pt] (B) at (1.9,-0.5){};
	\node[circle, fill, inner sep=.9pt, outer sep=0pt] (C) at (3,0.5){};
	\node[circle, fill, inner sep=.9pt, outer sep=0pt] (D) at (1.5,1.8){};
		\begin{scope}
			\draw [midarrow={>}] (A)--(B);
			\draw [midarrow={>}] (B)--(C);
			\draw [dashed,midarrow={>}] (A)--(C);	
			\draw [midarrow={>}] (A)--(D);
			\draw [midarrow={>}] (B)--(D);
			\draw [midarrow={>}] (D)--(C);
			\fill[blue, opacity=.5] (A) circle (3pt) node[black, font=\small, left]{\,$0$\,};
			\fill[blue, opacity=.5] (B) circle (3pt) node[black, font=\small, below right]{\,$1$\,};
			\fill[blue, opacity=.5] (C) circle (3pt) node[black, font=\small, right]{\,$3$\,};
			\fill[green, opacity=.5] (D) circle (3pt) node[black, font=\small, left]{\,$2$\,};
			\draw[thick,blue,blue] (A)--(B);
			\draw[thick,blue,blue] (B)--(C);
			\draw[dashed,thick,blue,blue] (A)--(C);			
		\end{scope}
\end{tikzpicture}

& \alpha_- (\lambda(0123))

\\
\\

\begin{tikzpicture}[thick][scale=0.7]
\fill[blue,blue, opacity=.5] (0,0)--(1.9,-0.5)--(3,0.5)--(0,0);
	\node[circle, fill, inner sep=.9pt, outer sep=0pt] (A) at (0,0){};
	\node[circle, fill, inner sep=.9pt, outer sep=0pt] (B) at (1.9,-0.5){};
	\node[circle, fill, inner sep=.9pt, outer sep=0pt] (C) at (3,0.5){};
	\node[circle, fill, inner sep=.9pt, outer sep=0pt] (E) at (1.5,-1.8){};
		\begin{scope}
			\draw [midarrow={>}] (A)--(B);
			\draw [midarrow={>}] (B)--(C);
			\draw [midarrow={>}] (A)--(C);	
			\draw [midarrow={>}] (A)--(E);
			\draw [midarrow={>}] (B)--(E);
			\draw [midarrow={>}] (C)--(E);
			\fill[blue, opacity=.5] (A) circle (3pt) node[black, font=\small, left]{\,$0$\,};
			\fill[blue, opacity=.5] (B) circle (3pt) node[black, font=\small, below left]{\,$1$\,};
			\fill[blue, opacity=.5] (C) circle (3pt) node[black, font=\small, right]{\,$3$\,};
			\fill[green, opacity=.5] (E) circle (3pt) node[black, font=\small, right]{\,$4$\,};
			\draw[thick,blue,blue] (A)--(B);
			\draw[thick,blue,blue] (B)--(C);
			\draw[thick,blue,blue] (A)--(C);			
		\end{scope}
\end{tikzpicture}

& \displaystyle \alpha_- (\lambda(0134) )

\\ 
\end{array} \right.$

\\
\end{tabular}
\end{center}

\noindent For which the local contribution to the state-sum is given by 

\[ \sum c_2^{-3}c_3^{-2} \prod {\rm dim}(\lambda(ij))
%
 ,
 \]
 
 \noindent in which the sum ranges over all colorings of the depicted simplexes and the product over the edges.
 
 Applying the analogue of Schur's lemma that holds by virtue of the simplicity of the edge labels gives
 
 \[ \sum c_2^{-3}c_3^{-2} \prod {\rm dim}(\lambda(ij)) %
  {\rm dim}(\lambda(03)) .\]

Note that this string diagram has two edges labeled $\lambda(03)$.  By virtue of the fact that there is only the zero 2-arrow between any pair of parallel non-isomorphic 1-arrows, these two can be allowed to be given different labels ($\lambda(03)$ on the upper and $\lambda(03)^\prime$ on the lower) summing over both, without changing the value of the state-sum.  Then contracting the edge of the string diagram labeled $\lambda(03)^\prime$ then gives

 \[ \sum c_2^{-3}c_3^{-2} \prod {\rm dim}(\lambda(ij)) %
. \]

Next, we have the following six tetrahedron states of 2-6 move.

%

The local state-sum for this configuration is then

\[ \sum c_2^{-4}c_3^{-2} \prod {\rm dim}(\lambda(ij)) %
\]

\noindent where again the sum is over all colorings of the local configuration and the product is over all of the edges occurring in the local configuration.  Applying the analogue of Schur's lemma rewrites this as

\[ \sum c_2^{-4}c_3^{-2} \prod {\rm dim}(\lambda(ij)) %
 \] 
\[ \times \,\, {\rm dim}(\lambda(0\frac{7}{2})) \,\,\,{\rm dim}(\lambda(1\frac{7}{2})) \,\,\,
{\rm dim}(\lambda(04)) \,\,\,{\rm dim}(\lambda(14)) . \]

\noindent Repeated use of the same trick of allowing the two instances of the colors occurring on two edges in the same connected string diagram separated by a portion representing an endomorphism to vary independently, then contracting one reduces this to

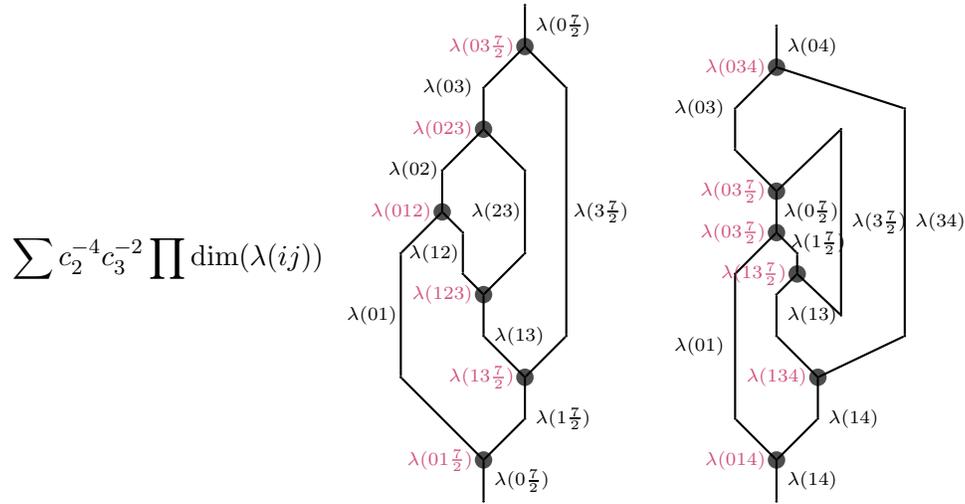
\[  \sum c_2^{-4}c_3^{-2} \prod {\rm dim}(\lambda(ij)) %
 .\]

Applying the analogue of Schur's lemma once more to insert the left diagram into the edge colored $\lambda(0\frac{7}{2})$ and repeatedly contracting edges then gives.

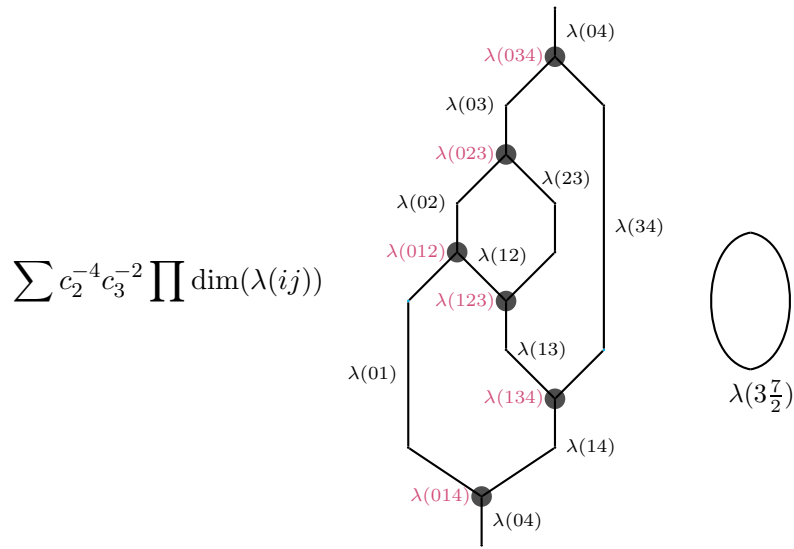
\vspace{4mm}
\[ \sum c_2^{-4}c_3^{-2} \prod {\rm dim}(\lambda(ij)) %
 .\]

\noindent Again, the sum is over colorings and the product is over the edges mentioned in the string diagram.  Note, however, that at this point something new occurred: the loop labeled $\lambda(3\frac{7}{2})$ has value ${\rm dim}(\lambda(3\frac{7}{2})$, so it, together with the factor of that same dimension, give become a factor of $c_2$, canceling one factor of $c_2$ and reducing the expression to that computed from the 2-simplex configuration, thus giving invariance under the 2-6 move.
    
    \item[3-6] move.  Here a subdivision of three 3-simplexes sharing an edge, lying in a 1-dimensional stratum, whose link is the three edges in the 3-dimensional stratum, is induced by an Alexander subdivision of the common edge.  There are again cases depending on how the strata induce the directed space structure and whether, a common face of a pair of the 3-simplexes lies in the 2-dimensional stratum (as in the case we explicitly compute) or not. 
    
    First, three tetrahedron state of 3-6 move is depicted below:



The corresponding local state sum is then

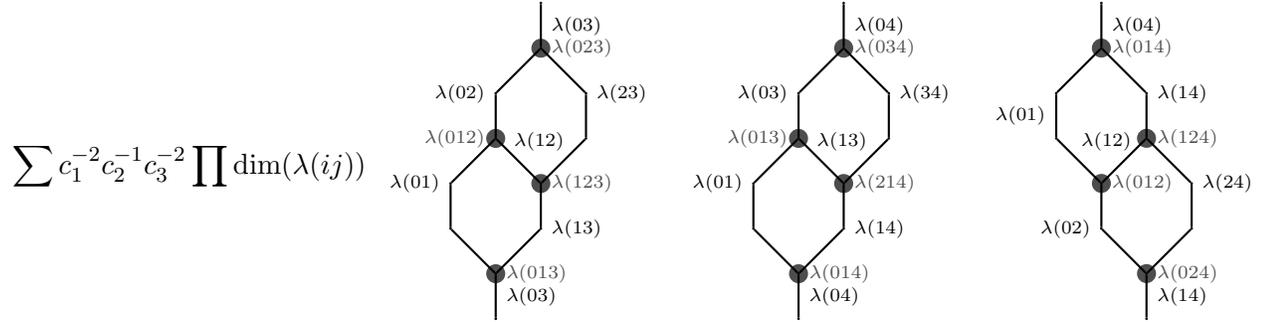
\[ \sum  c_1^{-2}c_2^{-1}c_3^{-2} \prod {\rm dim}(\lambda(ij)) %
 .\] 

Applying the analogue of Schur's lemma twice gives

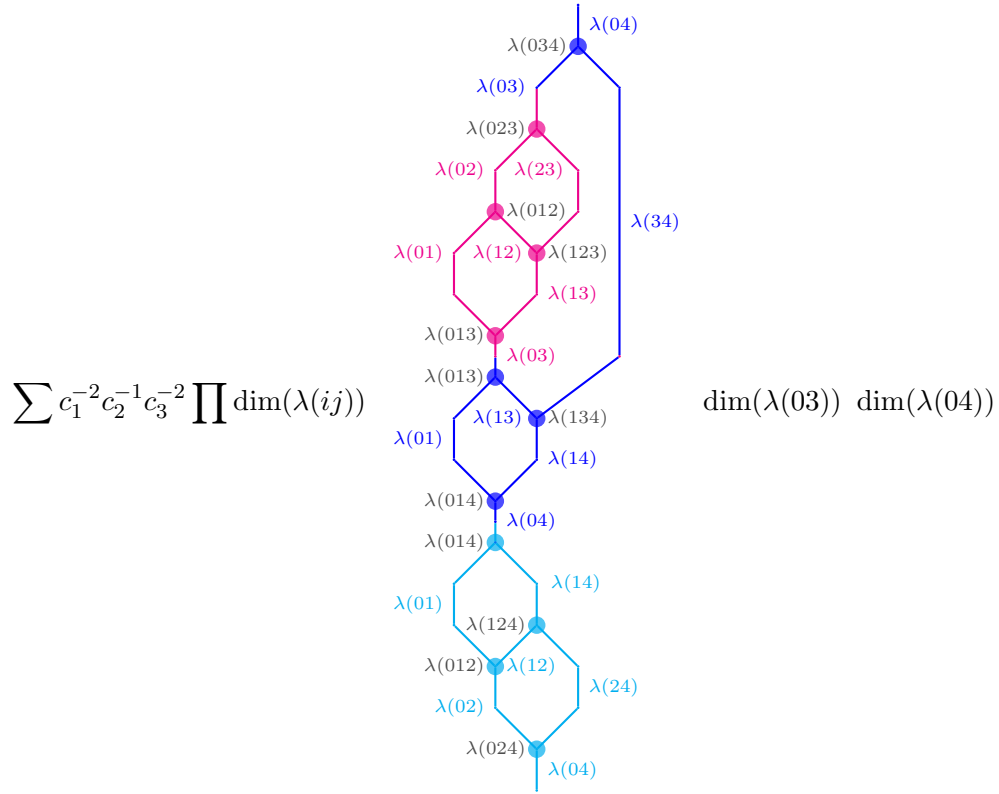
\[ \sum c_1^{-2}c_2^{-1}c_3^{-2} \prod {\rm dim}(\lambda(ij)) %
 {\rm dim}(\lambda(03)) \,\,\, {\rm dim}(\lambda(04)) \]

The same trick of letting the labels on pairs of edged with the same label vary independently, then contracting one of the edges then gives:

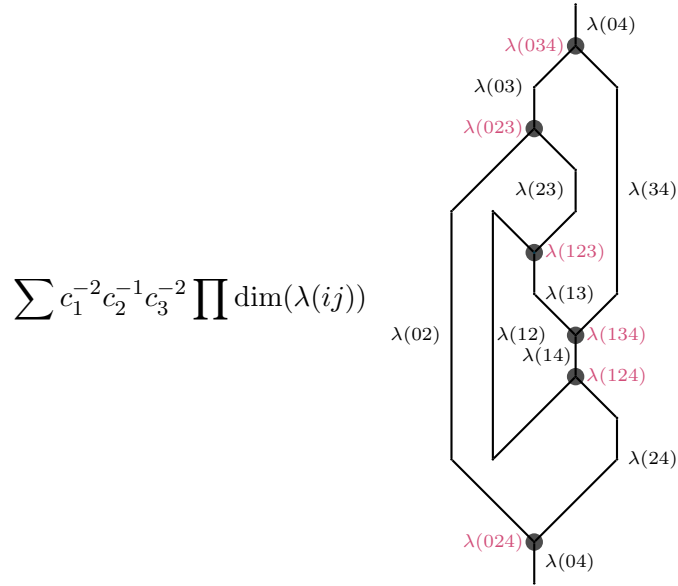
\[ \sum c_1^{-2}c_2^{-1}c_3^{-2} \prod {\rm dim}(\lambda(ij)) %
 \]

\vspace{4mm}
Next, we have the following six tetrahedron states of 3-6 move.

\vspace{4mm}
%

Here the local state-sum is given by


\[ \sum c_1^{-3}c_2^{-1}c_3^{-2} \prod {\rm dim}(\lambda(ij)) %
\]

Applying the analogue of Schur's lemma to insert the first (resp. second, third) diagram in the lower row on the edge labeled $\lambda(\frac{1}{2},3)$ (resp. $\lambda(\frac{1}{2},4)$, $\lambda(\frac{1}{2},4)$) on the corresponding diagram above, and contracting one of the edges labeled $\lambda(\frac{1}{2},3)$ (resp. $\lambda(\frac{1}{2},4)$, $\lambda(\frac{1}{2},4)$) then results in the sum

\[\sum c_1^{-3}c_2^{-1}c_3^{-2} \prod {\rm dim}(\lambda(ij)) %
.\]

Another application of the analogue of Schur's Lemma gives

\[\sum c_1^{-3}c_2^{-1}c_3^{-2} \prod {\rm dim}(\lambda(ij)) %
\]

\[ \times\,\, {\rm dim}(\lambda(04)) , \]

Contracting the middle edge labeled $\lambda(04)$ and then the resulting edge labeled $\lambda(\frac{1}{2} 4)$ (noting that this is the last edge so labeled so the dimension of the label no longer occurs in the product of dimensions) then gives

\[\sum c_1^{-3}c_2^{-1}c_3^{-2} \prod {\rm dim}(\lambda(ij)) %
\]

Applying Schur's lemma again gives

\[\sum c_1^{-3}c_2^{-1}c_3^{-2} \prod {\rm dim}(\lambda(ij)) %
\]

\[ \times\,\, 
{\rm dim}(\lambda(03)) .\]

Contracting in sequence the vertical edge labeled $\lambda(03)$, the resulting edge labeled $\lambda(\frac{1}{2}3)$ and the resulting edge labeled $\lambda(\frac{1}{2}1)$ (in the latter two cases eliminating the last instance of the label and the corresponding factor of its dimension from the coefficient), then gives

\[ \sum c_1^{-3}c_2^{-1}c_3^{-2} \prod {\rm dim}(\lambda(ij))%
 .\]

The loop, together with the corresponding factor of ${\rm dim}(\lambda(0 \frac{1}{2})$ in the coefficent, then sum to give an overall factor of $c_1$, which when cancelled against one factor of $c_1^{-1}$ leave the same result as we obtained from the three-simplex configuration, thus establishing invariance under the 3-6 move.

\end{itemize}
\end{proof}

\section{Directions for Future Research}

The constructions given in this paper, of course, call for the development of computational tools to actually compute examples.  They also suggest other avenues of research into the construction of TQFTs with defects-with-structure.  In no particular order:

\begin{itemize}
    \item[$\bullet$] Can they be extended to the study of general stratified spaces equipped with the exit path directed space structure?
    \item[$\bullet$] Can the assumption of quasi-sphericality be relaxed, either in general or in particular circumstances?
    \item[$\bullet$] What is the appropriate categorical structure to imitate the constructions in the present paper to give a Crane-Yetter \cite{CY2} type theory with defects-with-structure?  Here, of course, there are many possible types of defects, including knotted surfaces, with or without boundary, embedded 3-manifolds, again, with or without boundary, embedded curves, or even, subsuming them all, a stratification by submanifolds with boundary and corners.
\end{itemize}

\section{Acknowledgements}
The second author's part in the completion of this work took place while in residence at the Mathematical Sciences Research Institute, Berkeley, CA.  He wishes to thank MSRI for its hospitality and 
the National Science Foundation for financial support under grant DMS-0441170.

\end{document}